\documentclass[10pt,notitlepage,twoside,a4paper]{amsart}
 \usepackage{amsfonts}

\usepackage{amsmath,amssymb,enumerate}

\usepackage{epsfig,fancyhdr,color}%,showkeys,amsmidx

\usepackage{amssymb}
\usepackage{amsmath,amsthm}
\usepackage{latexsym}
\usepackage{amscd}
\usepackage{psfrag}
\usepackage{graphicx}
\usepackage[latin1]{inputenc}
\usepackage[all]{xy}
\usepackage[mathcal]{eucal}

\definecolor{NoteColor}{rgb}{1,0,0}

\renewcommand{\textsc}{\textcolor{red}}

\newtheorem{theorem}{\rm\bf Theorem}[section]
\newtheorem{proposition}[theorem]{\rm\bf Proposition}
\newtheorem{lemma}[theorem]{\rm\bf Lemma}
\newtheorem{corollary}[theorem]{\rm\bf Corollary}
\newtheorem*{theorem 1}{\rm\bf Proposition 1}
\newtheorem*{theorem 2}{\rm\bf Proposition 2}

\theoremstyle{definition}
\newtheorem{definition}[theorem]{\rm\bf Definition}

\theoremstyle{remark}

\begin{document}

\title[Homeomorphisms of the space of geodesic laminations]{\textbf{On the homeomorphisms of the space of geodesic laminations  on a hyperbolic surface}}

\author{C. Charitos}
\address{Charalampos Charitos,  Laboratory of Mathematics, Agricultural University of Athens, Iera Odos 75, 
 118 55 Athens, Greece}
\email{bakis@aua.gr}

\author{I. Papadoperakis}
\address{Ioannis Papadoperakis,  Laboratory of Mathematics, Agricultural University of Athens, Iera Odos 75, 
 118 55 Athens, Greece}
\email{papadoperakis@aua.gr}

\author{A. Papadopoulos}
\address{Athanase Papadopoulos,  Universit{\'e} de Strasbourg and CNRS,
7 rue Ren\'e Descartes,
 67084 Strasbourg Cedex, France}
\email{athanase.papadopoulos@math.unistra.fr}

\maketitle
\begin{abstract}
 We prove that for any orientable connected surface $S$ of finite type which is not a a sphere with at most four punctures or a torus with at most two punctures, the natural homomorphism from the extended mapping class group of $S$ to the group of  homeomorphisms of the space of geodesic laminations on $S$, equipped with the Thurston topology, is an isomorphism.
 \end{abstract}

\maketitle

%\bigskip

\medskip

\medskip

\section{Introduction}

In this paper $S=S_{g,p}$ is an orientable connected surface of finite type,
of genus $g\geq0$ with $p\geq0$ punctures. We assume that $S$ is
not a sphere with at most four punctures or a torus with at most two punctures.
Fixing a complete hyperbolic metric of finite area on $S$ we consider the set
$\mathcal{GL}(S)$ of geodesic laminations on $S$ with compact support.
As a set of compact subspaces of the metric space $S$, $\mathcal{GL}(S)$ is equipped with the Hausdorff metric  which we denote by $d_{H}$. We denote by
$\mathcal{T}_{H}$ the topology induced by $d_{H}$, and we call it the Hausdorff topology. We consider on
$\mathcal{GL}(S)$ a second topology $\mathcal{T}$, referred to as the Thurston topology  (see the definition in \S \ref{s:def} below) and which is weaker than 
$\mathcal{T}_{H}$. Any homeomorphism $h:S\rightarrow S$ induces by push-forward a map
$h_{\ast}:\mathcal{GL}(S)\rightarrow\mathcal{GL}(S)$ which is a homeomorphism
with respect to both topologies, $\mathcal{T}_{H}$ and $\mathcal{T}.$ The main
result of this paper is the following:

\begin{theorem} \label{main theorem} 

The natural homomorphism from the extended mapping class group of $S$ to the group of  homeomorphisms of $\mathcal{GL}(S)$ equipped with the Thurston topology, is an isomorphism. 

\end{theorem}

In particular, if $f:\mathcal{GL}(S)\rightarrow\mathcal{GL}
(S)\mathcal{\ }$is a homeomorphism with respect to the topology $\mathcal{T}$, then there is
a homeomorphism $h:S\rightarrow S$ such that $h_{\ast}=f.$

This shows that the space $\mathcal{GL}(S)$ equipped with the Thurston topology is  definitely not topologically homogeneous. Furthermore, since its homeomorphism group is countable, this space does not contain any open set which is a manifold of any positive dimension, in contrast with measured lamination space which is topologically a sphere. 

The analogous result for the Hausdorff topology is not true but it is conceivabe that  for every isometry $f$ of the
metric space $(\mathcal{GL}(S),d_{H})$, there is a homeomorphism
$h:S\rightarrow S$ such that $h_{\ast}=f.$

The result is in the spirit of several rigidity results that were obtained by
various authors in the context of mapping class group actions on different
spaces, and it is close more specially to the results in \cite{Papa}, \cite{Ohs} and
\cite{Ohs2} which concern actions by homeomorphisms. We note that although the statements of these and other rigidity theorems all look alike, all of them are interesting because each proof displays new features of the space on which the mapping class group acts, in different settings (combinatorial, topological, holomorphic, metric, etc.), and the arguments in each setting are often different.

The proof of Theorem \ref{main theorem} is given in \S \ref{s:proof}.  It involves the construction from the homeomorphism $f$ an automorphism of the complex of curves of $S$. The main technical point is to show that $f$ preserves the set of finite laminations. For the passage between properties of laminations and homeomorphisms of $\mathcal{GL}(S)$, the idea is to translate inclusions $\Lambda_1\subset \Lambda_2$ between geodesic laminations into set-theoretical properties between open sets of $\mathcal{GL}(S)$.

We would like to thank Ken'ichi Ohshika who read a preliminary version of this paper and corrected several mistakes, and Yi Huang  who pointed out a gap in a previous version.

\section{Definitions and Preliminaries}\label{s:def}

On the surface $S$, we fix a complete hyperbolic metric of
finite area. A \textit{geodesic lamination} $\Lambda\subset S$
is a compact non-empty subset which is the union of disjoint simple geodesics. We say that a geodesic
lamination is \textit{maximal} if it is not a proper sublamination of
any other geodesic lamination. We say that a geodesic lamination is \textit{minimal}
if it does not contain any proper sublamination. Note that this set-theoretic definition of a minimal lamination is not the same as the usual definition of a minimal laminations, since usually a lamination is said to be minimal if its leaves are dense in the support.

A geodesic lamination which is a finite union of geodesics is
called a \textit{finite lamination.} Otherwise, it is said to be \emph{infinite}. 

The following two subsets of $\mathcal{GL}(S)$ will play special roles in the sequel:

$\bullet $ $\mathcal{FGL}(S)$ is the subset of finite laminations of
$\mathcal{GL}(S)$ . An element of
$\mathcal{FGL}(S)$ is made out of a finite union of disjoint simple closed
geodesics $\{\gamma_{i}\}$ together with a finite number of infinite
geodesics, each spiraling from each end around one geodesic in $\{\gamma
_{i}\}$. 

$\bullet$  $\mathcal{CGL}(S)$ is the set of geodesic
laminations whose leaves are simple closed geodesics. 

Obviously $\mathcal{CGL}
(S)\subset\mathcal{FGL}(S).$

For any subset $X$ of $S$ and for any $\epsilon >0$, we set

\[N_{\varepsilon}(X)=\{x\in S:\exists y\in X \hbox{ with } 
 d(x,y)<\varepsilon\}.\] 
 
 The following definition is classical:

\begin{definition}
Let $X$ and $X^{\prime}$ be two compact subsets of $S$. The \emph{Hausdorff
distance} between $X$ and $X^{\prime}$  is the quantity 
$$d_{H}(X,X^{\prime})=\inf\{\varepsilon>0:X\subset N_{\varepsilon} 
(X^{\prime}) \hbox{ and } X^{\prime}\subset N_{\varepsilon}(X\}.$$
\end{definition}

It is easy to see that the function $d_H$ is a distance function on the set of compact subsets of $S$. Such a definition was made by F. Hausdorff for the set of compact subsets of   $\mathbb{R}^n$, and it was used by H. Busemann \cite{Busemann} for the set of compact sets of a general metric space.

We shall mostly use the notions $N_{\varepsilon}(X)$ and $d_H(X,Y)$ for elements $X,Y\subset S$ which are geodesic laminations on $S$. We also denote by $d_H$ the restriction of the Hausdorff metric to $\mathcal{GL}(S)$.

We also use the following notation:

For any $\Lambda$ $\in\mathcal{GL}(S)$ and for any $\varepsilon>0$, 

  \[\mathcal{V}_{\varepsilon}(\Lambda)=\{\Lambda^{\prime} 
\in\mathcal{GL}(S):N_{\varepsilon}(\Lambda^{\prime})\supset\Lambda \hbox{ and }
N_{\varepsilon}(\Lambda)\supset\Lambda^{\prime}\}.\]

The topology induced by $d_H$ on the set of subsets of $S$ as well as its restiction to $\mathcal{GL}(S)$, which we shall denote by $\mathcal{T}_{H}$, are
called the \textit{Hausdorff topology}. For the topology $\mathcal{T}_{H}$ any set
$\mathcal{V}_{\varepsilon}(\Lambda)$ is open. Moreover, it is easy to
see that the collection of sets $\mathcal{V}_{\varepsilon}(\Lambda)$ for all
$\varepsilon>0$ and $\Lambda$ $\in\mathcal{GL}(S),$ constitute a basis for
$\mathcal{T}_{H}.$

We now equip the set $\mathcal{GL}(S)$ with a second topology.

\begin{definition}
Let $V$ be an open subset of $S.$ Set 
$$\mathcal{U}_{V}=\{\Lambda \in\mathcal{GL}(S):\Lambda\cap V\neq\emptyset\}.$$
 We let $\mathcal{T}$ be 
the topology on $\mathcal{GL}(S)$ with subbasis the sets
$\mathcal{U}_{V},$ where $V$ varies over the set of open subsets of $S.$
\end{definition}

Following the
terminology of (\cite{CEG}, Def. I.4.1.10), we call the topology $\mathcal{T}$ on $\mathcal{GL}(S)$ the \textit{Thurston topology}. The original reference for the
topology $\mathcal{T}$ is (Thurston \cite{Thurston}, Section 8.10), where $\mathcal{T}
$ is referred to as the \textit{geometric topology}. Clearly the topology
$\mathcal{T}$ does not satisfy the first axiom of separation. Indeed, take a geodesic lamination $\Lambda$ that contains a strict sublamination $\Lambda_1\subsetneqq\Lambda_2$; then every open set for $\mathcal{T}$ containing $\Lambda_1$ contains $\Lambda$. In particular
the topology $\mathcal{T}$ is not Hausdorff, unlike the topology $\mathcal{T}_H$, which is induced by a metric.

\begin{lemma}
\label{weaker topology}$\mathcal{T}\subset\mathcal{T}_{H}$ i.e. the topology
$\mathcal{T}$ is weaker than the topology $\mathcal{T}_{H}.$
\end{lemma}

\begin{proof}
It suffices to prove that  for each open subset  $V$ of $S$, $\mathcal{U}_{V}\in\mathcal{T}_{H}$. For $\Lambda\in\mathcal{U}_{V}$, we have
$\Lambda\cap V\neq\emptyset.$ For any $x\in\Lambda\cap V$, there
exists an open ball $B(x,\varepsilon_{\Lambda})$ in $S$  of center $x$ and radius
$\varepsilon_{\Lambda}$ such that $B(x,\varepsilon_{\Lambda})\subset V$. We
now prove the following:

\begin{equation} \label{as}
\Lambda\in\mathcal{V}_{\varepsilon_{\Lambda}}(\Lambda)\subset
\mathcal{U}_{V}. 
\end{equation}

First, it is obvious that $\Lambda\in\mathcal{V}_{\varepsilon_{\Lambda}}
(\Lambda).$ To prove the inclusion $\mathcal{V}_{\varepsilon_{\Lambda}
}(\Lambda)\subset\mathcal{U}_{V}$ we note that if $\Lambda^{\prime}
\in\mathcal{V}_{\varepsilon_{\Lambda}}(\Lambda)$ then $\Lambda\subset
N_{\varepsilon_{\Lambda}}(\Lambda^{\prime})$ and hence $x\in N_{\varepsilon
_{\Lambda}}(\Lambda^{\prime}).$ Therefore we can find a point $y$ in $\Lambda^{\prime}$
such that $d(x,y)<\varepsilon_{\Lambda}$ and hence $y\in\Lambda^{\prime}\cap
B(x,\varepsilon_{\Lambda})\neq\emptyset.$ Since $B(x,\varepsilon_{\Lambda
})\subset V,$ this implies that $\Lambda^{\prime}\cap V\neq\emptyset.$
Therefore $\Lambda^{\prime}\in\mathcal{U}_{V}$ and the inclusion
$\mathcal{V}_{\varepsilon_{\Lambda}}(\Lambda)\subset\mathcal{U}_{V}$ is proven.

Now from (\ref{as}) we deduce immediately that 
$$\mathcal{U}_{V}
=\cup\{\mathcal{V}_{\varepsilon_{\Lambda}}(\Lambda):\Lambda
\in\mathcal{GL}(S) \hbox{ and } \Lambda\cap V\neq\emptyset\}.$$
 Therefore
$\mathcal{U}_{V}\in\mathcal{T}_{H}$. \end{proof}

The metric space $(\mathcal{GL}(S),d_{H})$ is compact. This is a general result on the Hausdorff metric on the set $\mathcal{B}(X)$ of compact subsets of a metric space $X$, and $\mathcal{GL}(S)$ is a closed subset of $\mathcal{B}(S)$
(cf. \cite{Casson},Theorem 3.4 for this special case). Therefore $\mathcal{T}_{H}$ is a compact topology.
Since $\mathcal{T\subset T}_{H},$ it follows that $\mathcal{T}$ is also a compact topology.

In the next theorem we summarize basic properties of minimal geodesic
laminations and we also give a description of the structure of maximal
geodesic laminations. The properties are all well known from Thurston's theory, and we give references for the convenience of the reader.

\begin{theorem}
\label{structure} {\rm  (I) (Proposition A.2.1, p. 142 in \cite{Otal})} Let
$\Lambda$ be an arbitrary geodesic lamination of $S.$ Then $S-\Lambda$
consists of finitely many components. Let $U$ be such a component. The
completion $C(U)$ of $U$ with respect to the metric induced by the Riemannian
metric of $S$ is a complete hyperbolic surface of finite area with geodesic boundary.

{\rm (II) (Corollary A.2.4, p. 143 in \cite{Otal} or Lemmata 4.2 and 4.3 in
\cite{Casson})} Let $\Lambda$ be a minimal geodesic lamination with infinitely
many leaves. Then every leaf of $\Lambda$ is dense in $\Lambda.$
Furthermore, $\Lambda$ contains a finite number of leaves which are isolated
from one side. These leaves appear as boundary geodesics of some $C(U),$ where
$U$ is a component of $S-\Lambda$; they will be referred to as \emph{boundary leaves} of
$\Lambda.$

{\rm (III) (Theorem I.4.2.8, p. 83 in \cite{CEG})} Let $\Lambda$ be an arbitrary geodesic
lamination of $S$. Then $\Lambda$ consists of the disjoint union of a finite
number of minimal sublaminations of $\Lambda$ together with a finite set of
additional geodesics each end of which spirals onto a minimal lamination. Each
of the additional geodesics is isolated, i.e. it is contained in an open subset
of $S$ which is disjoint from the rest of the lamination.
\end{theorem}

In Item (III) above, the fact that an end of a geodesic \emph{spirals} on a minimal sublamination $\Lambda'$ of $\Lambda$ means that  the set of accumulation points of this end on the surface is $\Lambda'$. 

The following lemma and proposition will also be used in the next section.

\begin{lemma}
\label{2 boundary leaves}Let $\Lambda$ be an infinite minimal geodesic
lamination of $S.$ Then $\Lambda$ has at least two boundary leaves.
\end{lemma}

\begin{proof}The lift of $\Lambda$ to the universal cover of $S$ is a geodesic lamination $\hat{\Lambda}$ of the hyperbolic plane. Each component of the complement of $\hat{\Lambda}$ is an ideal polygon. The images in $S$ of two boundary leaves of such a polygon give the desired boundary leaves of $\Lambda$.
\end{proof}
\begin{proposition}[see \cite{CEG}, I 4.2.14 p. 81] \label{dense} The finite laminations are dense in the space of geodesic lamnitations equipped with the Hausdorff topology. Hence the same holds for the Thurston topology. 
\end{proposition}

\section{On the action of a homeomorphism of $\mathcal{GL}(S)$}

We denote by $\mathcal{O}(S)$ the set of open subsets of $S$ and we fix an element $\Lambda$ of $\mathcal{GL}(S)$. We consider the sets $$\mathcal{O}_{\Lambda}(S)=\{V\in 
 \mathcal{O}(S)\ :\  V\cap\Lambda\neq\emptyset\}\subset  \mathcal{O}(S)$$ and 
 $$\mathcal{U}(\Lambda
)=\cap_{V \in\mathcal{O}_{\Lambda}(S)}  \mathcal{U}_{V}\subset \mathcal{GL}(S).$$
We have the following.

\begin{lemma}
\label{inclusion equivalence (1)} Let $\Lambda_{1},\Lambda_{2}\in
\mathcal{\ GL}(S).$ Then $\Lambda_{1}\subset\Lambda_{2}$ if and only if
$\Lambda_{2}\in\mathcal{U}(\Lambda_{1}).$
\end{lemma}

\begin{proof}
Assume that $\Lambda_{1}\subset\Lambda_{2}.$ For any $V\in$ $\mathcal{O}_{\Lambda
_{1}}(S),$ we consider the set $\mathcal{U}_{V}.$ Then $\Lambda_{2}\cap
V\neq\emptyset$, hence $\Lambda_{2}\in\mathcal{U}_{V}.$ Therefore
$\Lambda_{2}\in\mathcal{U}(\Lambda_{1}).$

Conversely, assume $\Lambda_{2}\in\mathcal{U}(\Lambda_{1}).$ If $\Lambda_{1} $ is
not a subset of $\Lambda_{2},$ then there exists $x\in\Lambda_{1}$ with
$x\notin\Lambda_{2}.$ Since $\Lambda_{2}$ is a compact subset of $S,$ there
exists $\varepsilon>0$ such that the open ball $B(x,\varepsilon)$ does not intersect $\Lambda_{2}.$ This implies that
$\Lambda_{2}\notin\mathcal{U}_{B(x,\varepsilon)}.$ This is a contradiction
since $\Lambda_{2}\in\mathcal{U}(\Lambda_{1})$. Hence $\Lambda_1\subset \Lambda_2$. 
\end{proof}

\begin{lemma}
\label{inclusion equivalence (2)} Let $\Lambda_{1},\Lambda_{2}\in
\mathcal{\ GL}(S).$ Then $\Lambda_{1}\subset\Lambda_{2}$ if and only if
$\mathcal{U} (\Lambda_{2})\subset\mathcal{U}(\Lambda_{1}).$
\end{lemma}

\begin{proof}
First assume that $\Lambda_{1}\subset\Lambda_{2}$ and let $\Lambda
^{\prime}\in\mathcal{U}(\Lambda_{2}).$ Then $\Lambda^{\prime}\cap
V\neq\emptyset$ for each $V\in\mathcal{O}(S)$ with $V\cap\Lambda_{2}
\neq\emptyset.$ Therefore $\Lambda^{\prime}\cap V\neq\emptyset$ for each
$V\in\mathcal{O}(S)$ with $V\cap\Lambda_{1}\neq\emptyset.$ Therefore
$\Lambda^{\prime}\in\mathcal{U}(\Lambda_{1}).$

Now assume that $\mathcal{U}(\Lambda_{2})\subset\mathcal{U}(\Lambda
_{1}).$ Obviously, $\Lambda_{2}\in\mathcal{U}(\Lambda_{2})$ and hence $\Lambda_{2}\in
\mathcal{U}(\Lambda_{1}).$ From Lemma \ref{inclusion equivalence (1)}, this
implies that $\Lambda_{1}\subset\Lambda_{2}.$
\end{proof}

\begin{lemma}
\label{equality}Assume that $f$ is a homeomorphism of $\mathcal{GL}(S)$ with
respect to the Thurston topology. If $\Lambda\in\mathcal{GL}(S),$ then $f(
\mathcal{U}(\Lambda))=\mathcal{U}(f(\Lambda)).$
\end{lemma}

 \begin{proof}
 
 It suffices to show that $$ \mathcal{U}(f(\Lambda))\subset f(\mathcal{U}(\Lambda)) \ \forall \Lambda\in \mathcal{GL}(S).$$
   Indeed, this implies 
 \[\mathcal{U}(f^{-1}(\Lambda))\subset f^{-1}(\mathcal{U}(\Lambda)),\]
 which implies 
 \[f(\mathcal{U}(f^{-1}(\Lambda))\subset \mathcal{U}(\Lambda),\]
 which implies   
 \[f(\mathcal{U}(\Lambda_0)\subset \mathcal{U}(f(\Lambda_0),\]
  where $\Lambda_0=f^{-1}(\Lambda)$. Thus, we get $f(\mathcal{U}(\Lambda_0)\subset \mathcal{U}(f(\Lambda_0)$ for any $\Lambda_0$ in $\mathcal{GL}(S)$. 
 
 To prove  the assertion, we start by
 \[f(\mathcal{U}(\Lambda))=f(\cap\{\mathcal{U}_\Lambda : V\in \mathcal{O}_\Lambda (S)\})=\cap
 \{f(\mathcal{U}_V) : V\in \mathcal{O}_\Lambda(S)\}
 .\]
  Now let $V_0$ be an arbitrary element of $\mathcal{O}_\Lambda(S)$. Then, $\Lambda$ is in $\mathcal{U}_{V_{0}}$ and $f(\Lambda)$ is in $f(\mathcal{U}_{V_{0}})$.
 But  $f(\mathcal{U}_{V_{0}})\in\mathcal{T}$ (that is, $f(\mathcal{U}_{V_{0}})$ is an open set of Thurston's topology). 
 Therefore, there exist $V_1,\ldots, V_n\in \mathcal{O}_{f(\Lambda)}(S)$ such that \[\displaystyle \cap_{1\leq i\leq n}\mathcal{U}_{V_{i}}\subset f(\mathcal{U}_{V_{0}}).\]
 Therefore, we have 
\[\cap \{\mathcal{U}_V : V\in \mathcal{O}_{f(\Lambda)}(S)\}\subset f(\mathcal{U}_{V_{0}}),\]
 which implies that for every $V_0\in \mathcal{O}_\Lambda(S)$, we have \[\mathcal{U}(f(\Lambda))\subset f(\mathcal{U}_{V_{0}}).\]
 Finally, we obtain
 \[\mathcal{U}(f(\Lambda))\subset \cap \{f(\mathcal{U}_V) : V\in \mathcal{O}_\Lambda(S)\}\]
 which implies $\mathcal{U}(f(\Lambda))\subset f(\mathcal{U}(\Lambda))$.
  \end{proof}

From the above lemmata, we obtain the following:

\begin{corollary}\label{cor:inclusion}
\label{inclusion}Assume that $f$ is a homeomorphism of $\mathcal{GL}(S)$ with
respect to the Thurston topology and let $\Lambda_{1},\Lambda_{2}
\in\mathcal{GL}(S).$ Then $\Lambda_{1}\subset\Lambda_{2}$ implies that
$f(\Lambda_{1})\subset f(\Lambda_{2}).$
\end{corollary}

\begin{proof}
From Lemma \ref{inclusion equivalence (2)}, the inclusion $\Lambda_{1}
\subset\Lambda_{2}$ implies that $\mathcal{U}(\Lambda_{2})\subset
\mathcal{U}(\Lambda_{1}).$ Since $f$ is a bijection, we have $f(\mathcal{U}(\Lambda_{2}))\subset
f(\mathcal{U}(\Lambda_{1})).$ From Lemma \ref{equality} we get
$\mathcal{\ U}(f(\Lambda_{2}))\subset\mathcal{U}(f(\Lambda_{1}))$ which
implies again, by Lemma \ref{inclusion equivalence (2)}, that $f(\Lambda
_{1})\subset f(\Lambda_{2}).$
\end{proof}

\begin{lemma}
\label{maximal}Let $f$ be a homeomorphism of $\mathcal{GL}(S)$ with respect to
the Thurston topology. Then $f$ sends a maximal (respectively minimal) geodesic lamination of $S $ to
a maximal (respectively minimal) geodesic lamination of $S.$
\end{lemma}

\begin{proof}
Let $\Lambda$ be a maximal geodesic lamination of $S.$ If $f(\Lambda)$ is not
maximal then there is a maximal geodesic lamination $\Theta$ such that $f(\Lambda
)\subsetneqq\Theta.$ Let $\Lambda^{\prime}=f^{-1}(\Theta).$ Then, by
Corollary \ref{inclusion}, we have $\Lambda\subsetneqq\Lambda^{\prime}$
which contradicts the maximality of $\Lambda$.

Likewise, let $\Lambda$ be a minimal geodesic lamination of $S$. If $f(\Lambda)$ is not minimal, then there exists a lamination $\Theta$ such that $\Theta \subsetneqq f(\Lambda)$. Let $\Lambda'= f^{-1}(\Theta)$. Then, from Corollary \ref{inclusion} again, we have 
$\Lambda' \subsetneqq \Lambda$, which contradicts the minimality of $\Lambda$.

\end{proof}

From Theorem \ref{structure} (III), every $\Lambda
\in\mathcal{GL}(S)$ has a finite number of sublaminations. Thus, we give the
following definition.

\begin{definition}
Let $\Lambda\in\mathcal{GL}(S).$ A \emph{chain of sublaminations} of $\Lambda$ is a
finite sequence $(\Lambda_{i}),$ $i=0,1,..,n$ of sublaminations of $\Lambda$
such that $\emptyset\neq\Lambda_{n}\subsetneqq\Lambda_{n-1}\subsetneqq
...\subsetneqq\Lambda_{1}\subsetneqq\Lambda_{0}=\Lambda.$ We denote such a chain by
$\mathcal{C}_{\Lambda}.$ The integer $n$ will be called the \emph{length} of
$\mathcal{C}_{\Lambda}$ and will be denoted by $l(\mathcal{C}_{\Lambda}).$

A chain of sublaminations $\mathcal{C}_{\Lambda}$ will be called \emph{maximal} if
its length is maximal among all chains of sublaminations of $\Lambda$.

The length of a maximal chain of sublaminations $\mathcal{C}_{\Lambda}$ of
$\Lambda$ depends only on $\Lambda.$ Therefore the number $l(\mathcal{C}
_{\Lambda})$ will be referred to as the length of $\Lambda$ and will be denoted by
$\mathrm{length}(\Lambda).$
\end{definition}

\begin{lemma}
\label{length} Let $f$ be a homeomorphism of $\mathcal{GL}(S)$ with respect to
the Thurston topology and let $\Lambda_{n}\subsetneqq\Lambda_{n-1}\subsetneqq
...\subsetneqq\Lambda_{1}\subsetneqq\Lambda_{0}=\Lambda$ be a maximal chain of
sublaminations of $\Lambda.$ Then $f(\Lambda_{n})\subsetneqq f(\Lambda
_{n-1})\subsetneqq...\subsetneqq f(\Lambda_{1})\subsetneqq f(\Lambda
_{0})=f(\Lambda)$ is a maximal chain of sublaminations of $f(\Lambda)$ and
$\mathrm{length}(\Lambda_{k})=\mathrm{length}(f(\Lambda_{k}))$ for each $k=0,1,..,n.$
\end{lemma}

\begin{proof}
The proof follows immediately from Corollary \ref{inclusion}.
\end{proof}

 We call a \emph{generalized pair of pants} a hyperbolic surface which is homeomorphic to a sphere with three holes, a hole being either a geodesic boundary component or a cusp. 
 
Now we can prove the following proposition.

\begin{proposition}
\label{finite to finite}
Let $f$ be a homeomorphism of $\mathcal{GL}(S)$ with
respect to the Thurston topology. Then,
\begin{enumerate}
\item \label{item:1} $f$ sends any maximal finite lamination which contains a collection of curves that decompose $S$ into generalized pair of pants to a maximal finite lamination that has the same property.

\item \label{item:2}   $f$ sends any laminations whose leaves are all closed to a lamination having the same property. Furthermore if such a
  lamination $\Lambda$ has $k$ components then $f(\Lambda)$ has also $k$ components.
  \item \label{item:3} $f$ sends finite laminations to finite laminations.
  \end{enumerate}
\end{proposition}

\begin{proof}
(\ref{item:1}) Let $\Lambda$ be a maximal geodesic lamination. From Theorem \ref{structure},
$S-\Lambda$ consists of finitely many open components. 

\textit{Claim 1: Let $U$ be a component of $S-\Lambda$. Then, the completion }$C(U)$\textit{\ of }
$U$\textit{\ is a hyperbolic surface which is isometric either to a hyperbolic ideal
triangle or to a surface of genus 0 with a cusp and an open geodesic as
boundary. The latter will be referred to as a cusped hyperbolic monogon ; it is obtained from a hyperbolic ideal triangle by gluing together two sides of this ideal triangle.}

\noindent \textit{Proof of Claim 1. }The surface $C(U)$ is a complete hyperbolic surface
of finite area. Therefore\textit{\ }$C(U)$ has finitely many boundary components which
are either closed geodesics or open geodesics. We may easily verify that if
$C(U)$ is not of the type described in the claim then the lamination $\Lambda$
would not be maximal because we could add open geodesics $l_{i}$
to $\Lambda$ and construct a lamination $\Lambda^{\prime}\supsetneqq\Lambda.$
This proves Claim 1.

Now let $\lambda_{1},...,\lambda_{n}$ be the leaves of $\Lambda$ which are boundary geodesics of the completion $C(U)$ of some component $U $ of
$S-\Lambda.$ The leaves $\lambda_{i}$ can be of two types:

(i)  an open geodesic of $S$ which is an isolated leaf of $\Lambda$;

(ii) an open geodesic of $S$ which is a leaf isolated from one side in $\Lambda$.

A geodesic of type (ii) appears as a boundary leaf of a minimal infinite sublamination
of $\Lambda.$ Among the $\{\lambda_{i}\}$, we may assume, without loss of generality,
that $\lambda_{1},...,\lambda_{k}$ are the isolated leaves of $\Lambda,$ for
$0\leq k\leq n.$ From Theorem \ref{structure}, it follows that if
$\Lambda$ is a finite maximal geodesic lamination then $k=n$ and that if $\Lambda$
is a maximal infinite geodesic lamination which is also minimal then $k=0.$

To the leaves $\lambda_{1},...,\lambda_{n}$ of $\Lambda$ we add the leaves
$c_{1},...,c_{m}$ of $\Lambda$ which are simple closed geodesics, if there
exist any. Let $A_{\Lambda}=\{\lambda_{1},..,\lambda_{n},c_{1},..,c_{m}\}$. 

We define a \emph{generating set} for a geodesic lamination $\Lambda$ to be a set $\mathcal{A}=\{\mu_1,\ldots,\mu_k\}$ of leaves of $\Lambda$ such that the union of the closures of the leaves that belong to $\mathcal{A}$ is the lamination $\Lambda$. It follows from Theorem \ref{structure} that every geodesic lamination on $S$ has a finite generating set.

Our terminology is motivated by Claim 2 that follows now.

\textit{\noindent Claim 2: The set }$A_{\Lambda}=\{\lambda
_{1},..,\lambda_{n},c_{1},..,c_{m}\}$\textit{\ is a generating set for }$\Lambda
$\textit{\ and any proper sublamination of }$\Lambda$\textit{\ is the union of closures of leaves that belong to some proper subset of }$A_{\Lambda}.$

\noindent\textit{Proof of Claim 2}. The claim follows immediately from the definition of $A$ and from 
Theorem \ref{structure} (III).

It is well-known and easy to see, using an Euler characteristic count, that the maximum number of pairwise disjoint simple closed
geodesics in $S$ is equal to $3g-3+b$ and that these geodesics cut $S$ into $2g-2+b$ hyperbolic generalized pairs of pants. It is also easy to see that the maximum number of open geodesics $l_{i}$
of $S$ which decompose $S$ into hyperbolic ideal triangles is equal to $6g-6+3b$. 

\textit{\noindent Claim 3: If }$\Lambda$\textit{\ is a maximal 
finite geodesic lamination that contains a generalized pair of pants decomposition $P$ then a maximal chain } \textit{\ of
sublaminations of }$\Lambda$\textit{\ has length }$9g-9+3b.$

\noindent\noindent\textit{Proof of Claim 3. }The lamination $\Lambda$ contains  $3g-3+b$ simple closed geodesics, say $c_{1},..,c_{3g-3+b},$ which cut $S$
into generalized pairs of pants, and additional open isolated geodesics, say $\lambda_{1},...,\lambda_{r},$ such that
each $\lambda_{i}$ spirals about some $c_{j}$. We may
add to $\lambda_{i}$ open geodesics $\lambda_{k}^{\prime},$ which from one
direction abut to a cusp and from the other direction spiral about some
$c_{j},$ such that all the geodesics $\lambda_{i}$ and $\lambda_{i}^{\prime}$
decompose $S$ into hyperbolic ideal triangles. Since the total number of $\lambda_{i}$ and $\lambda_{i}^{\prime}$ is equal to $6g-6+3b$ we deduce that
the number $r$ of $\lambda_{i}$ is equal to $6g-6+2b.$ Therefore the set
$A_{\Lambda}=\{\lambda_{1},..,\lambda_{6g-6+2b},c_{1},..,c_{3g-3+b}\}$ is a
generating set of $\Lambda.$ From Claim 2, a chain of sublaminations
$\Lambda_{n}\subsetneqq\Lambda_{n-1}\subsetneqq...\subsetneqq\Lambda
_{1}\subsetneqq\Lambda_{0}=\Lambda$ is maximal if and only if for each
$k=1,..,n,$ $\Lambda_{k}-\Lambda_{k-1}$ is a single leaf belonging to
$A_{\Lambda}.$ Such a sequence $\Lambda_{i}$ can be constructed as follows:
From $\Lambda_{0}=\Lambda$ we first remove, one by one, all leaves
$\lambda_{i};$ after removing all these leaves we continue removing, one by
one, all leaves $c_{j}.$ Obviously $n=9g-9+3b$.

\noindent\textit{Claim 4:\ Let }$\Lambda$\textit{\ be a maximal finite 
geodesic lamination which does not contain a generalized pants decomposition. Then $\mathrm{length}(\Lambda) <9g-9+3b$.}

\noindent\textit{Proof of Claim 4}. Since $\Lambda$ does not contain a generalized pants decomposition it follows that if  $c_1,\ldots, c_k$ are the closed geodesics in $\Lambda$, then  $k$ is strictly smaller than $3g-3+b$. Now as in the proof of Claim 3, we may prove that $\Lambda$ contains additional open geodesics $\lambda_1,\ldots,\lambda_{6g-6+2b}$ which spiral about the $c_j$'s. Obviously, a maximal chain $\mathcal{C}_{\Lambda}$ of sublaminations of $\Lambda$ has length less than $9g-9+3b$.

\noindent\textit{Claim 5:\ Let }$\Lambda$\textit{\ be a maximal infinite 
geodesic lamination. Then $\mathrm{length}(\Lambda) <9g-9+3b$.}

Consider a maximal chain of sublaminations
of $\Lambda,$ say $\Lambda_{k}\subsetneqq\Lambda_{k-1}\subsetneqq
...\subsetneqq\Lambda_{1}\subsetneqq\Lambda_{0}=\Lambda.$ Let $A_{\Lambda
}=\{\lambda_{1},..,\lambda_{n},c_{1},..,c_{m}\}$ be a generating set of
$\Lambda$, where each $\lambda_i$ is an open geodesic and each $c_i$ is a closed geodesic, and we assume that the generating set  $A_{\Lambda}$ is \emph{minimal} in the sense that no proper subset of $A_{\Lambda}$ is a generating set of $\Lambda$. 

First, we prove that $m<3g-3+b$. We know that $m\leq3g-3+b.$ If
$m=3g-3+b$ then the closed geodesics cut $S$ into generalized
pairs of pants. This implies that a minimal sublamination, say $\Lambda
^{\prime},$ of $\Lambda$ with infinitely may leaves must be in the
interior of a generalized pair of pants. But it is easy to
see that such a sublamination $\Lambda^{\prime}$ does not exist. Therefore $m<3g-3+b.$

Second, we prove that $n\leq$ $6g-6+2b.$ By Theorem
\ref{structure}, the leaves $\lambda_{1},..,\lambda_{n}$ are either isolated
open leaves or boundary leaves. Since $\Lambda$ is maximal, for each component
$U$ of $S-\Lambda$ the completion $C(U)$ is either a hyperbolic ideal triangle
or a cusped hyperbolic monogon. The area of every such surface is
equal to $\pi$, therefore, from the Gauss-Bonnet theorem the number of
components $U$ is equal to $4g-4+2b.$ Now, each hyperbolic ideal triangle has
three sides and each cusped hyperbolic monogon has one side. On the
other hand the number of cusped hyperbolic monogons is equal to
$b.$ Therefore the total number of sides of $C(U)$ is $3(4g-4+2b)-2b=12g-12+4b.$ Now if a leaf $\lambda_{i}$ is isolated it belongs
exactly to two components $C(U).$ If a leaf $\lambda_{i} $ is not isolated then it belongs to a minimal sublamination of $\Lambda$, say $K_i$. From Lemma \ref{2 boundary leaves}, $K_i$ contains at least two boundary leaves and we take $\lambda_i$ to be one of them. Therefore, among all the sides of the given $C(U)$, at least two of them are boundary leaves of $K_i$. 
 This implies that $n\leq\frac{12g-12+4b}{2}=6g-6+2b.$ Assume that
$\lambda_{1},...,\lambda_{r}\in A_{\Lambda}$, $0\leq r<n$ are isolated leaves
and $\lambda_{r+1},...,\lambda_{n}\in A_{\Lambda}$ are isolated from one side.
Every $\lambda_{s}$ with $s>r$ belongs to a unique minimal infinite
sublamination $\Lambda_{s}^{\prime}$ of $\Lambda.$

Now the maximal chain of sublaminations $\Lambda_{k}\subsetneqq\Lambda_{k-1}\subsetneqq
...\subsetneqq\Lambda_{1}\subsetneqq\Lambda_{0}=\Lambda$ of
$\Lambda$ is constructed as follows: We pass from $\Lambda_{i-1}$ to
$\Lambda_{i}$ by removing all the isolated open geodesics
$\lambda_{i}\in A_{\Lambda}$ with $i\leq r$, then  every minimal infinite
sublamination $\Lambda_{t}^{\prime}$ of $\Lambda$ and every closed
geodesic $c_{j}\in A_{\Lambda}.$ Obviously the length of $\Lambda$ is less
than $9g-9+3b.$ This finishes the proof of Claim 5.

Now, from Lemma \ref{length} the homeomorphism $f$ preserves the length of a
maximal lamination $\Lambda$ and this finishes the proof of (\ref{item:1}).

Now we prove (\ref{item:2}). First, we claim that if $K$ is a geodesic lamination consisting of $k$ closed geodesics then $f(K)$ is a finite lamination. To prove this, we consider a maximal finite geodesic lamination $\Lambda$ containing a generalized pair of pants  decomposition $P$ with $K\subset P$. By  (\ref{item:1}), $f(\Lambda)$ is a maximal finite lamination and thus our claim follows. From this claim and Lemma \ref{length}, it follows that if $K_1=\{c\}$ is a geodesic lamination consisting of a single closed geodesic then the image of $K_1$ by $f$ consists of a single closed geodesic. 

Now we have proved Statement (\ref{item:2}) for $n=1$ and we proceed by induction. We assume that the statement is true for all $n\leq k$ and we prove it for $n=k+1$.

Assume that $K$ consists of $k+1$ closed geodesics. If our statement were not true for $n=k+1$, then from our inductive assumption and Lemma \ref{length}, $f(K)$ would consist of $k$ closed geodesics, say $c_1,\cdots, c_k$, plus one open geodesic. Then, any sublamination of $K$ consisting of $k$ geodesics would be sent to $\{c_1,\cdots, c_k\}$. But this is impossible since $f$ is a homeomorphism. This completes the induction and the proof of (\ref{item:2}).

 Now we prove (\ref{item:3}). Let $K$ be a finite lamination. Obviously, a minimal sublamination of $K$ consists of a single closed geodesic. Assume that $f(K)$ is not finite. Then there is an infinite minimal sublamination $\Lambda_0$ of $f(K)$. From Lemma \ref{maximal}, $\Lambda_0$ is the image of a minimal sublamination $K_0$ of $K$, via $f$, i.e. $f(K_0)=\Lambda_0$. But from   (\ref{item:2}), $f(K_0)$ consists of a single closed geodesic. Therefore we have a contradiction, which proves  (\ref{item:3}). 
\end{proof}

\bigskip\bigskip

\section{Proof of Theorem \ref{main theorem}} \label{s:proof}

We first prove that if $f:\mathcal{GL}(S)\rightarrow\mathcal{GL}
(S)\mathcal{\ }$is a homeomorphism with respect to the topology $\mathcal{T}$, then there is
a homeomorphism $h:S\rightarrow S$ such that $h_{\ast}=f.$

 We need the following lemma.

\begin{lemma}\label{lemma1}
\label{Haus-Th}Assume that $(\Lambda_{n})$ is a sequence of laminations that
converges to a lamination $\Lambda$ with respect to the Hausdorff topology.
Assume also that $(\Lambda_{n})$ converges to $\Lambda^{\prime} $ with respect
to the Thurston topology. Then $\Lambda^{\prime}\subset\Lambda.$
\end{lemma}

\begin{proof}
If $\Lambda^{\prime}$ is not contained in $\Lambda$ then there exists a point
$x\in\Lambda^{\prime}-\Lambda.$ Therefore there exists an open ball
$B(x,2\varepsilon)$ in $S$ of center $x$ and radius $2\varepsilon$ such that
$B(x,2\varepsilon)\cap\Lambda=\emptyset.$ We also have $\Lambda^{\prime
}\in\mathcal{U}_{B(x,\varepsilon)}.$

Now, since $\Lambda_{n}\rightarrow\Lambda$ with respect to the Hausdorff
topology we deduce that there is $n_{0}\in\mathbb{N}$ such that the following holds:

\begin{equation} \label{ast}
\Lambda_{n}\cap B(x,\varepsilon)=\emptyset \ \forall n\geq n_{0}.
\end{equation}
On the other hand, $\Lambda_{n}\rightarrow\Lambda^{\prime}$ with respect to
the Thurston topology. This implies that there exists $n_{1}\in\mathbb{N} $
such that $\Lambda_{n}\in\mathcal{U}_{B(x,\varepsilon)}$ for $n\geq n_{1}$ and
hence $\Lambda_{n}\cap B(x,\varepsilon)\neq\emptyset$ for $n\geq n_{1}.$ This
gives a contradiction to the relation (\ref{ast}) which implies that
$\Lambda^{\prime}\subset\Lambda.$
\end{proof}

Now note that there is a natural identification between the
subset $\mathcal{CGL}(S)$ and the complex of curves $\mathcal{C}(S)$ of $S.$
Therefore, from Proposition \ref{finite to finite}, $f$ induces an automorphism on $\mathcal{C} 
(S).$ From \cite{Ivanov}, \cite{Kork} and \cite{Luo} and under the hypothesis of Theorem \ref{main theorem}, we obtain
a homeomorphism $h:S\rightarrow S$ such that $h_{\ast}=f$ on $\mathcal{CGL} 
(S).$  

We now prove
that $h_{\ast}=f$ on $\mathcal{FGL}(S).$ Composing $f$ with $(h_{\ast})^{-1},$ if necessary, it suffices to assume
that $f=\mathrm{id}$ on $\mathcal{CGL}(S)$ and prove that $f=\mathrm{id}$ on
$\mathcal{FGL}(S).$ To do it, we first prove two lemmas:

\begin{lemma}  \label{lemma2} Let $\Lambda =\{\gamma ,\gamma _{1},\gamma _{2}\}$\ be a
geodesic lamination consisting of two disjoint simple closed geodesics $
\gamma _{1},\gamma _{2}$\ and one open geodesic $\gamma $\ spiraling in one
direction about $\gamma _{1}$\ and in the other direction about $\gamma
_{2}. $\ Then $f(\Lambda )=\Lambda .$
\end{lemma}

\begin{proof}
We take a generalized pair of pants decomposition $P$ and a maximal finite
geodesic lamination $\Lambda _{P}$ such that

(1) $\Lambda \subset \Lambda _{P};$

(2) for each component $R$ of $S-P$ and every pair $\lambda _{1}\neq \lambda
_{2}$ of boundary geodesics of $R$ there is a leaf $\lambda $ of $\Lambda
_{P}$ contained in $R$ which spirals in one direction about $\lambda _{1}$
and in the other direction about $\lambda _{2};$

(3) for every component $\gamma $ of $P,$ the leaves of $\Lambda _{P}$
spiraling about $\gamma $ from different sides of $\gamma $ induce opposite
orientations on $\gamma .$

Such a lamination $\Lambda _{P}$ can be approximated in the Hausdorff
topology and hence in the Thurston topology by a sequence $C_{n}$ of simple
closed geodesics, where the $C_{n}$ are viewed as elements of $\mathcal{CGL}
(S).$ This implies that $f(\Lambda _{P})=\Lambda _{P}.$ Indeed, $
C_{n}\rightarrow \Lambda _{P}$ with respect to the Hausdorff topology and
hence with respect to the Thurston topology. Therefore $f(C_{n})\rightarrow
f(\Lambda _{P})$ with respect to the Thurston topology. This implies, since $
f$ is the identity on $\mathcal{CGL}(S),$ that $C_{n}\rightarrow f(\Lambda
_{P})$ with respect to the Thurston topology. From Lemma \ref{lemma1} we
deduce that $f(\Lambda _{P})\subset \Lambda _{P}.$ But from Lemma \ref{length} the laminations $\Lambda _{P}$ and $f(\Lambda _{P})$ have the same length.
Therefore $f(\Lambda _{P})$ cannot have a smaller number of leaves than $
\Lambda _{P}.$ Therefore $f(\Lambda _{P})=\Lambda _{P}.$

Let us set now $f(\Lambda )=\Lambda ^{\prime }.$ From Proposition \ref{finite to finite} and Corollary \ref{cor:inclusion}, we deduce that $\Lambda ^{\prime }=\{\gamma
^{\prime },\gamma _{1},\gamma _{2}\}$ where $\gamma ^{\prime }$ spirals
about $\gamma _{1}$ and about $\gamma _{2}$. We shall prove that $\gamma
^{\prime }=\gamma .$

We can find a surface $Q\subset S$ with geodesic boundary containing the
lamination $\Lambda $ and which is the closure of a component of $S-P$ such
that $Q$ is one of the following types:

(i) a torus with one boundary component, with $\partial Q=\gamma_{1}$ or $
\partial Q=\gamma_{2};$

(ii) a pair of pants with no cusps, with $\gamma_{1}\cup\gamma_{2}\subset
\partial Q;$

(iii) a generalized pair of pants with a single cusp with $\partial Q=\gamma
_{1}\cup \gamma _{2}.$

First we show that $\gamma ^{\prime }\subset Q.$ Let $P_{\gamma }=P\cup
\{\gamma \}\subset \Lambda _{P}.$ Then, it is not hard to show, using
Corollary \ref{cor:inclusion} and Proposition \ref{finite to finite}, that $f(P_{\gamma
})=P\cup \gamma ^{\prime }.$ Therefore $P\cap \{\gamma ^{\prime
}\}=\emptyset .$ This implies that if $\gamma ^{\prime }$ is not contained
in $Q$ then it is contained in some generalized pair of pants $Q^{\prime }$
with $\gamma _{1},\gamma _{2}\in \partial Q^{\prime }$ and $\mathrm{Int}(Q)\cap
\mathrm{Int}(Q^{\prime })=\emptyset .$ Thus, we may find a simple closed geodesic $
\gamma _{0}$ of $S$ which intersects only one of the geodesics $\gamma ,$ $
\gamma ^{\prime }$ and $\gamma _{0}\cap \gamma _{i}=\emptyset $ for $i=1,2.$
To prove the last statement we need the assumption on the topological type
of $S$ that we made in the introduction. Assume without loss of generality
that $\gamma \cap \gamma _{0}=\emptyset $ and $\gamma ^{\prime }\cap \gamma
_{0}\neq \emptyset .$ From Corollary \ref{cor:inclusion} the lamination $\Lambda
_{1}=\{\gamma ,\gamma _{1},\gamma _{2},\gamma _{0}\}$ is sent to a
lamination $\Lambda _{1}^{\prime }$ which should contain the geodesics $
\gamma ^{\prime },\gamma _{1},\gamma _{2},\gamma _{0}.$ But this is
impossible since $\gamma ^{\prime }\cap \gamma _{0}\neq \emptyset $ and
hence $Q^{\prime }=Q.$

Now in the cases (ii) and (iii) above there is a unique leaf of $\Lambda
_{P} $ contained in $Q$, that spirals about $
\gamma _{1}$ and $\gamma _{2}$, namely the leaf $\gamma$. This implies that $\gamma ^{\prime }=\gamma $
since $f(\Lambda )=\Lambda ^{\prime }\subset \Lambda _{P}=f(\Lambda _{P})$
and $\gamma ^{\prime }\subset Q.$

Consider now the case (i) and without loss of generality we assume that $
\partial Q=\gamma _{2}.$ In this case, the lamination $\Lambda _{P}$ which
contains $\Lambda $ and which satisfies the requirements (1)-(3) above,
satisfies $\Lambda _{P}\cap Q=\{\gamma _{1},\gamma _{2},\gamma ,\delta
_{1},\delta _{2}\},$ where $\gamma $ is a geodesic spiraling about $\gamma
_{1}$ in both directions but from different sides of $\gamma _{1}$ and $
\delta _{1},$ $\delta _{2}$ are spiraling about $\gamma _{1}$ and $\gamma
_{2}.$ Denote by $\Lambda _{Q}$ the lamination $\Lambda _{P}\cap Q.$

Also, if $\gamma $ is an open simple geodesic of $S$, we denote by $
\overline{\gamma }$ the closure of $\gamma $ in $S.$ Obviously $\overline{
\gamma }$ is a geodesic lamination and consists of $\gamma $ and two or one
additional simple closed geodesics on which $\gamma $ is spiraling about. We note that we can talk about $f(\{
\overline{\gamma }\})$ and  not about $f(\{
 \gamma\})$ since $\overline{\gamma}$ is a lamination but $\gamma$ is not.  

\textit{Claim 1: }$f(\Lambda _{Q})=\Lambda _{Q}$\textit{\ and }$f(\{
\overline{\gamma }\})=\{\overline{\gamma }\}.$

\noindent \textit{\noindent Proof of Claim 1. }Since $f(\Lambda _{P})=\Lambda _{P}$\
and $f$\ is the identity on $\mathcal{CGL}(S)$\ we deduce that $f(\Lambda
_{Q})=\Lambda _{Q}.$\ On the other hand $f(\{\overline{\gamma }\})=\{
\overline{\gamma }\}$\ since $\gamma $\ is the unique leaf of $\Lambda _{Q}$
\ spiraling about $\gamma _{1}$\ in its both directions.

Let now $\beta _{1}$ be a simple open geodesic in $Q$ spiraling in both
directions about $\gamma _{1}$ from the same side of $\gamma _{1}$ and such
that $\beta _{1}\cap \delta _{1}=\emptyset .$

\textit{Claim 2: }$f(\{\overline{\beta _{1}}\})=\{\overline{\beta _{1}}\}.$

\noindent \textit{Proof of Claim 2}. Since $f(\{\overline{\gamma }\})=\{\overline{
\gamma }\},$\ it is easy to prove that either $f(\{\overline{\beta _{1}}
\})=\{\overline{\beta _{1}}\}$\ or $f(\{\overline{\beta _{1}}\})=\{\overline{
\beta _{2}}\},$\ where $\beta _{2}$ is the unique geodesic in $Q$\ spiraling
about $\gamma _{1}$\ in both directions and from the same side of $\gamma
_{1}$\ and such that $\beta _{2}\cap \gamma =\emptyset .$\ Now, we may find
an open geodesic $\alpha ,$\ as well as, a closed geodesic $\gamma
_{1}^{\prime },$\ such that:

$\gamma _{1}^{\prime }\cap Q=\emptyset ;$

$\alpha $\ is spiraling about $\gamma _{1}$\ and $\gamma _{1}^{\prime };$\

$\alpha \cap \beta _{1}=\emptyset $\ but $\alpha \cap \beta _{2}\neq
\emptyset .$

\noindent From (ii) we have that $f(\{\overline{\alpha }\})=\{\overline{
\alpha }\}$\ and therefore we deduce that $f(\{\overline{\beta _{1}}\})=\{
\overline{\beta _{1}}\}$ and Claim 2 is proved.

We are now able to prove that $f(\{\overline{\delta _{i}}\})=\{\overline{
\delta _{i}}\},$ $i=1,2.$ Indeed, considering the lamination $\Lambda _{Q}$
it is clear from Claim 1, that $f(\{\gamma _{1},\gamma _{2},\delta
_{1},\delta _{2}\})=\{\gamma _{1},\gamma _{2},\delta _{1},\delta _{2}\}.$
Now we consider the lamination $\Lambda _{0}=\{\gamma _{1},\gamma _{2},\beta
_{1},\delta _{1}\}.$ Since $f(\Lambda _{0})\subset Q,$ $f(\{\overline{\beta
_{1}}\})=\{\overline{\beta _{1}}\}$ and $\beta _{1}\cap \delta _{2}\neq
\emptyset $ we deduce easily that $f(\{\overline{\delta _{1}}\})=\{\overline{
\delta _{1}}\}$ and this completes the proof of Lemma \ref{lemma2}.
\end{proof}

\begin{lemma} \label{lemma3} Let $\Lambda =\{\gamma ,\gamma _{1}\}$  be
the geodesic lamination consisting of one simple closed geodesic $\gamma
_{1}$ and one open geodesic $\gamma $ spiraling in both
directions about $\gamma _{1}$. Then $f(\Lambda )=\Lambda$.
\end{lemma}

\begin{proof}
First we find a surface $Q\subset S$ with geodesic boundary containing $
\gamma $ and $\gamma _{1}$ such that $Q$ is one of the following surfaces:

(i) a torus with one boundary component such that $\gamma _{1}$ is not the
boundary of $Q;$

(ii) a torus with one boundary component such that $\gamma _{1}=\partial Q;$

(iii) a generalized pair of pants with cusps such that $\gamma _{1}\subset
\partial Q.$

In Case (i) we distinguish the following three subcases:

(i$_{a})$ $\gamma $ spirals in both directions about $\gamma _{1}$ from
different sides of $\gamma _{1}$ and induces opposite orientation on $\gamma
_{1}$ from each side;

(i$_{b})$ $\gamma $ spirals in both directions about $\gamma _{1}$ from the
same side of $\gamma _{1};$

(i$_{c})$ $\gamma $ spirals in both directions about $\gamma _{1}$ from
different sides of $\gamma _{1}$ and induces the same orientation on $\gamma
_{1}$ from each side.

Then, as in the proof of Lemma \ref{lemma2}, we prove that $\gamma \subset
Q. $

Assume that $f(\Lambda )=\Lambda ^{\prime }.$ From Proposition  \ref{finite to finite}
and Corollary \ref{cor:inclusion}, we deduce that $\Lambda ^{\prime }=\{\gamma
^{\prime },\gamma _{1}\},$ where $\gamma ^{\prime }$ spirals in both
directions about $\gamma _{1}.$

In Case (iii) we proceed as in Lemma \ref{lemma2}. More precisely, we
consider a lamination $\Lambda _{P}$ which satisfies the requirements
(1)-(3) of Lemma \ref{lemma2} and can check that there is a unique leaf of $
\Lambda _{P}$ contained in $Q,$ namely the leaf $\gamma ,$ which spirals
about $\gamma _{1}.$ As in Lemma \ref{lemma2}, this implies that $\gamma
^{\prime }=\gamma .$ Therefore, $\Lambda =\Lambda ^{\prime }.$

In Case (ii) we consider a simple closed geodesic $\gamma _{2}$ in $Q$ such
that $\gamma _{2}\cap \gamma =\emptyset .$ We also consider a simple open
geodesic $\delta $ in $Q$ spiraling about $\gamma _{1}$ in one direction and
about $\gamma _{2}$ in the other direction and such that $\{\gamma
_{1},\gamma _{2},\gamma ,\delta \}$ is a geodesic lamination. Then, from
Lemma \ref{lemma2} we have that $f(\{\overline{\delta }\})=\{
\overline{\delta }\}$ which implies that $f(\{\overline{\gamma }\})=\{
\overline{\gamma }\}.$

Cases (i$_{a})$ and (i$_{b})$ have been studied respectively in Claims 1 and
2 above. In each case we proved that $f(\{\overline{\gamma }\})=\{
\overline{\gamma }\}.$

In Case (i$_{c})$ we consider the lamination $\{\gamma _{1},\gamma
_{2},\gamma ,\delta _{1},\delta _{2}\}$, where $\delta _{1}$ and $\delta
_{2} $ are geodesics in $Q$ spiraling both in one direction about $\gamma
_{1}$ and in the other direction about $\gamma _{2}=\partial Q.$ Now, it is
easy to see that $f(\{\overline{\gamma }\})$ is a lamination consisting of $
\gamma _{1}$ and one open leaf spiraling in both directions about $\gamma
_{1}.$ On the other hand, from Lemma \ref{lemma2} we have that $f(\{
\overline{\delta _{i}}\})=\{\overline{\delta _{i}}\},$ $i=1,2.$ Hence we
deduce that $f(\{\overline{\gamma }\})=\{\overline{\gamma }\}$ and the lemma
is proved.
\end{proof}

Now we can prove that $f=\mathrm{id}$ on $\mathcal{FGL}(S).$ Indeed, consider
$K\in\mathcal{FGL}(S)$ such that $f(K)=K^{\prime}\neq K.$ Then, without loss
of generality, we may assume that there is a leaf $\gamma$ of $K$ such that
$\gamma$ is not a leaf of $K^{\prime}.$ First we remark that $\gamma$ cannot
be a closed geodesic since, if $\gamma$ is closed, $f(\{\gamma\})=\{\gamma\}$ and by Corollary
\ref{cor:inclusion}  $f(\{\gamma\})=\{\gamma\}\subset f(K).$ Let $\gamma$ be an
open geodesic and assume that $\gamma$ spirals about two disjoint closed
geodesics $\gamma_{1},$ $\gamma_{2}.$ Consider the lamination $\{\gamma,\gamma_{1} 
,\gamma_{2}\}.$ By Lemma \ref{lemma2}, $f(\{\gamma,\gamma_{1},\gamma
_{2}\})=\{\gamma,\gamma_{1},\gamma_{2}\}$ and hence $\{\gamma,\gamma
_{1},\gamma_{2}\}\subset f(K)$ by Corollary \ref{cor:inclusion}. Therefore
$\gamma$ is a leaf of $f(K),$ a contradiction which implies that $f(K)=K.$ (The case where $\gamma$ spirals about a
single closed geodesics $\gamma_{1}$ in both directions is treated similarly
using Lemma \ref{lemma3}.) 

Finally we will show that $f=id$ on $\mathcal{GL}(S).$ Let $\Lambda
\in\mathcal{GL}(S).$ Then from Proposition \ref{dense} such a lamination $\Lambda$
can be approximated in the Hausdorff topology and hence in the Thurston
topology by a sequence $F_{n}$ of finite laminations. This implies that
$f(\Lambda)=\Lambda.$ Indeed, $F_{n}\rightarrow\Lambda$ with respect to the
Hausdorff and hence with respect to the Thurston topology. Therefore
$f(F_{n})=F_{n}\rightarrow f(\Lambda)$ with respect to the Thurston topology.
From Lemma \ref{Haus-Th} we deduce that $f(\Lambda)\subset\Lambda.$ Both
$\Lambda$ and $f(\Lambda)$ are infinite laminations. Therefore, from Theorem
\ref{structure} (III), $\Lambda$ and $f(\Lambda)$ consist of the disjoint union of a
finite number of infinite minimal sublaminations with a finite set of
isolated, open geodesics each end of which spiral onto a minimal
sublamination. Now, if $f(\Lambda)\neq\Lambda$ then $f(\Lambda) $ must consist
either of a smaller number of minimal sublamination of $\Lambda$ or/and of a
smaller number of isolated, open geodesics. This implies that \textrm{length} 
$(f(\Lambda))<\mathrm{length}(\Lambda)$ which contradicts Lemma \ref{length}.
Therefore $f(\Lambda)=\Lambda$. 

This proves that any homeomorphism $f:\mathcal{GL}(S)\rightarrow\mathcal{GL}
(S)$ with respect to the topology $\mathcal{T}$ is induced by a
a homeomorphism $h:S\rightarrow S$.

Thus, the natural homomorphism from the extended mapping class group $\Gamma^*(S)$ of $S$ to the group of  homeomorphisms of the space $\mathcal{GL}(S)$, equipped with the Thurston topology, is an isomorphism. To complete the proof of Theorem \ref{main theorem},  we need to show that if two elements of $\Gamma^*(S)$  have the same action on $\mathcal{GL}(S)$, then they are equal.  Under the hypothesis of the theorem, and if we furthermore exclude the case of a closed surface of genus 2, this result follows from the fact that the homomorphism from $\Gamma^*(S)$ to the automorphism group of the complex of curves $C(S)$ of $S$ is injective (see \cite{Ivanov}). It remains to consider the case where $S$ is a closed surface of genus 2.

Thus, we now assume that $S$ is a closed surface of genus 2. It is known that in this case if $g$ is an element of $\Gamma^*(S)$ that induces the identity map on $C(S)$, then $g$ is either the identity element or a hyperelliptic involution. It remains to show that a hyperelliptic involution does not induce the identity on $\mathcal{GL}(S)$. To see this, let $\iota$ denote the hyperelliptic involution. We note that there is a pair of pants decomposition $P$ of $S$ that is invariant by $\iota$. By using various ways in which an infinite geodesic in a pair of pants spirals along the boundary components, we can complete the three curves in $P$ into a maximal geodesic lamination $\Lambda$ whose image by $\iota$ is a lamination that is different from $\Lambda$. This shows that $\iota$ does not induce the identity map on $\mathcal{GL}(S)$. This completes the proof of Theorem \ref{main theorem}.


\begin{thebibliography}{99}

\bibitem{Busemann} H. Busemann, \emph{The geometry of geodesics}, Acad. Press, 1955.

\bibitem{CEG}R. Canary, D. Epstein, P. Green, \textit{Fundamentals of
Hyperbolic Manifolds: Selected Expositions,} LMS, Lecture Note Series 328,
Cambridge University Press, 2006.

\bibitem{Casson}A. Casson, S. Bleiler, \textit{Automorphisms of Surfaces after
Nielsen and Thurston}, London Math. Society, Student Texts 9, 1988.

\bibitem{Ivanov}N. V. Ivanov, Automorphisms of Teichm\"{u}ller modular groups,
Lecture Notes in Math., No. 1346, Springer-Verlag, Berlin and New York, 1988, 199--270.


\bibitem{Kork}M. Korkmaz, \textit{Automorphisms of the complex of curves on
punctured spheres and punctured tori, }Topology and its Appl. 95(2), 85-111, 1999.

\bibitem{Luo}F. Luo, \textit{Automorphisms of the complex of curves},
Topology, 39(2), 283-298, 2000.


\bibitem{Mata}K. Matsuzaki and M. Taniguchi, \textit{Hyperbolic Manifolds and
Kleinian Groups}, Oxford Science Publications, 1998.

\bibitem{Ohs}K. Ohshika, Reduced Bers boundaries of Teichm\"{u}ller spaces,
preprint, 2011, arxiv:1103.4680v2.


\bibitem{Ohs2}K. Ohshika,  A note on the rigidity of unmeasured lamination space, preprint, 2011.

\bibitem{Otal}J. P. Otal, \textit{Le th\'{e}or\`{e}me d'hyperbolization pour
les vari\'{e}t\'{e}s fibr\'{e}es de dimension 3}, Soc. Math. France,
Ast\'{e}risque 235, 1996.

\bibitem{Papa}A. Papadopoulos, \textit{A rigidity theorem for the mapping
class group action on the space of unmeasured foliations on a surface},
Proceedings of AMS, 136, p. 4453-4460, 2008.

\bibitem{Thurston}W. Thurston, \textit{The Geometry and Topology of Three
Manifolds,} Princeton Lecture Notes, 1979.
 \end{thebibliography}
\end{document}